\documentclass{amsproc}
\usepackage{amsfonts}

\theoremstyle{plain}

\newtheorem{corollary}{Corollary}

\newtheorem{theorem}{Theorem}
\numberwithin{equation}{section}
\input{tcilatex}

\begin{document}
\title[Euler and Bernoulli polynomials, Pascal matrix]{A few remarks on
Euler and Bernoulli polynomials and their connections with binomial
coefficients and modified Pascal matrices}
\author{Pawe\l\ J. Szab\l owski}
\address{Department of Mathematics and Information Sciences,\\
Warsaw University of Technology\\
ul. Koszykowa 75, 00-662 Warsaw, Poland}
\email{pawel.szablowski@gmail.com}
\date{June, 2013}

\begin{abstract}
We prove certain identities involving Euler and Bernoulli polynomials that
can be treated as recurrences. We use these and also other  known identities
to indicate connection of Euler and Bernoulli numbers with entries of
inverses of certain lower triangular built of binomial coefficients. Another
words we interpret Euler and Bernoulli numbers in terms of modified Pascal
matrices.
\end{abstract}

\maketitle

\section{Introduction and notation\label{intr}}

The aim of the paper is to indicate close relationship between Euler and
Bernoulli polynomials and certain lower triangular matrices with entries
depending on binomial coefficients and some other natural numbers. In this
way we point out new interpretation of Euler and Bernoulli numbers.

In the series of papers \cite{Zhang97}, \cite{Zhang98}, \cite{Kim00}, \cite%
{Zhang07} Zhang, Kim and their associates have studied various
generalizations of Pascal matrices and examined their properties. The
results of this paper can be interpreted as the next step in studying
properties of various modifications of Pascal matrices. 

We do so by studying known and indicating new identities involving Euler and
Bernoulli polynomials. One of them particularly simple involves these
polynomials of either only odd or only even degrees.

More precisely we will express these numbers as entries of inverses of
certain matrices build of almost entirely of binomial coefficients.

Throughout the paper we will use the following notation. Let a sequences of
lower triangular matrices $\left\{ A_{n}\right\} _{n\geq 0}$ be such that $%
A_{n}$ is $(n+1)\times (n+1)$ matrix and matrix say $A_{k}$ is a submatrix
of every $A_{n},$ for $n\geq k.$ Notice that the same property can be
attributed to inverses of matrices $A_{n}$ (of course if they exist) and to
products of such matrices. Hence to simplify notation we will denote entire
sequence of such matrices by one symbol. Thus e.g. sequence $%
\{A_{n}\}_{n\geq 0}$ will be dented by $\mathcal{A}$ or $[a_{ij}]$ if $a_{ij}
$ denotes $(i,j)-$th entry of matrices $A_{n},$ $n\geq i.$ The sequence $%
\{A_{n}^{-1}\}_{n\geq 0}$ will be denoted by $\mathcal{A}^{-1}$ or $%
[a_{ij}]^{-1}.$ Analogously if we have two sequences say $\mathcal{A}$ and $%
\mathcal{B}$ then $\mathcal{AB}$ would mean sequence $\left\{
A_{n}B_{n}\right\} _{n\geq 0}.$ It is easy to notice that all such lower
triangular matrices form a non-commutative ring with unity. Moreover this
ring is also a linear space over reals as far as ring's addition is
concerned. Diagonal matrix $\mathcal{I}$ with $1$ on the diagonal is this
ring's unity.

Let us consider also $(n+1)$ vectors $\mathbf{X}^{(n)}\allowbreak \overset{df%
}{=}(1,x,\ldots ,x^{n})^{T},$ $\mathbf{f(X)}^{(n)}\overset{df}{=}%
(1,f(x),\ldots ,f(x)^{n}),$ By $\mathbf{X}$ or by $[x^{i}]$ we will mean
sequence of vectors $\left( \mathbf{X}^{(n)}\right) _{n\geq 0}$ and by $%
\mathbf{f(X)}$ or by $[f(x)^{i}]$ the sequence of vectors $(\mathbf{f(X)}%
^{(n)})_{n\geq 0}.$

Let $E_{n}(x)$ denote $n-$th Euler polynomial and $B_{n}(x)$ $n-$th
polynomial. Let us introduce sequences of vectors $\mathbf{E}%
^{(n)}(x)\allowbreak =\allowbreak (1,2E_{1}(x),\ldots ,2^{n}E_{n}(x))^{T}$
and $\mathbf{B}^{(n)}(x)\allowbreak =\allowbreak (1,2B_{n}(x),\ldots
,2^{n}B_{n}(x))^{T}.$ These sequences will be denoted briefly $\mathbf{E}$
and $\mathbf{B}$ respectively. 

$\left\lfloor x\right\rfloor $ will denote so called 'floor' finction of $x$
that is the largest integer not exceeding $x.$

Since we will use in the sequel often Euler and Bernoulli polynomials we
will recall now briefly the definition of these polynomials. Their
characteristic functions are given by formulae (23.1.1) of \cite{Gradshtein}
and respectively:%
\begin{eqnarray}
\sum_{j\geq 0}\frac{t^{j}}{j!}E_{j}(x)\allowbreak  &=&\allowbreak \frac{%
2\exp (xt)}{\exp (t)+1},  \label{chE} \\
\sum_{j\geq 0}\frac{t^{j}}{j!}B_{j}(x)\allowbreak  &=&\allowbreak \frac{%
t\exp (xt)}{\exp (t)-1}.  \label{chB}
\end{eqnarray}

Numbers $E_{n}\allowbreak =\allowbreak 2^{n}E_{n}(1/2)$ and $%
B_{n}\allowbreak =\allowbreak B_{n}(0)$ are called respectively Euler and
Bernoulli numbers.

By standard manipulation on characteristic functions we obtain for example
the following identities some of which are well known $\forall k\geq 0:$ 
\begin{eqnarray}
2^{k}E_{k}(x)\allowbreak &=&\allowbreak \sum_{j=0}^{k}\binom{k}{j}%
E_{k-j}\times (2x-1)^{j},  \label{ide} \\
B_{k}(x)\allowbreak &=&\allowbreak \sum_{j=0}^{k}\binom{k}{j}B_{k-j}\times
x^{j},\text{~}x^{k}\allowbreak =\allowbreak \sum_{j=0}^{k}\binom{k}{j}\frac{1%
}{k-j+1}B_{j}(x),  \label{idb} \\
E_{k}(x) &=&\sum_{j=0}^{k}\binom{k}{j}2^{j}B_{j}(x)\frac{%
(1-x)^{n-j+1}-(-x)^{n-j+1}}{(n-j+1)},  \label{idbe} \\
E_{k}(x) &=&\sum_{j=0}^{k}\binom{k}{j}2^{j}B_{j}(\frac{x}{2})\frac{1}{k-j+1},
\label{ebyb} \\
B_{k}(x) &=&\sum_{j=0}^{k}\binom{k}{j}2^{j}B_{j}(x)\frac{%
(1-x)^{n-j}+(-x)^{n-j}}{2},  \label{idbb} \\
B_{k}(x) &=&2^{k}B_{k}(\frac{x}{2})+\sum_{j=1}^{k}\binom{k}{j}%
2^{k-j-1}B_{k-j}(\frac{x}{2}),  \label{bbyb}
\end{eqnarray}%
which are obtained almost directly from the following trivial identities
respectively:%
\begin{eqnarray*}
\frac{2\exp (xt)}{\exp (t)+1} &=&\frac{1}{\cosh (t/2)}\times \exp (\frac{t}{2%
}(2x-1)), \\
\frac{t~\exp (xt)}{\exp (t)-1} &=&\frac{t}{\exp (t)-1}\times \exp (xt),~\exp
(xt)\allowbreak =\frac{t~\exp (xt)}{\exp (t)-1}\times \frac{\exp (t)-1}{t},
\\
\frac{2\exp (xt)}{\exp (t)+1} &=&\frac{2t\exp (2xt))}{\exp (2t)-1}\times
(\exp ((1-x)t)-\exp (-xt))/t, \\
\frac{2\exp (xt)}{\exp (t)+1} &=&\frac{2t\exp (\frac{x}{2}(2t)))}{\exp (2t)-1%
}\times (\exp (t)-1)/t, \\
\frac{t~\exp (xt)}{\exp (t)-1} &=&\frac{2t\exp (2xt))}{\exp (2t)-1}\times
(\exp ((1-x)t)+\exp (-xt))/2, \\
\frac{t~\exp (xt)}{\exp (t)-1} &=&\frac{2t\exp (\frac{x}{2}2t))}{\exp (2t)-1}%
\times (\exp (t)+1)/2.
\end{eqnarray*}

By direct calculation one can easily check that:%
\begin{equation*}
\lbrack \binom{i}{j}]^{-1}=[(-1)^{i-j}\binom{i}{j}],~[\lambda ^{i-j}\binom{i%
}{j}]^{-1}=[(-\lambda )^{i-j}\binom{i}{j}],
\end{equation*}%
for any $\lambda .$ The above mentioned identities are well known. They
expose properties of Pascal matrices discussed in \cite{Zhang97}. Similarly
by direct application of (\ref{idb}) we have:%
\begin{equation}
\lbrack \binom{i}{j}\frac{1}{i-j+1}]^{-1}=[\binom{i}{j}B_{i-j}]  \label{imB}
\end{equation}
giving new interpretation of Bernoulli numbers. Now notice that we can
multiply both sides of (\ref{idb}) by say $\lambda ^{k}$ and define new
vectors $[(\lambda x)^{i}]$ and $[\lambda ^{i}B_{i}(x)].$ Thus (\ref{imB})
can be trivially generalized to 
\begin{equation*}
\lbrack \binom{i}{j}\frac{\lambda ^{i-j}}{i-j+1}]^{-1}=[\binom{i}{j}\lambda
^{i-j}B_{i-j}]
\end{equation*}%
for all $\lambda \in \mathbb{R}$, presenting first of the series of
modifications of Pascal matrices and their properties that we will present
in the sequel.

To find inverses of other matrices built of binomial coefficients we will
have to refer to the results of the next section.

\section{Main results\label{main}}

\begin{theorem}
$\forall n\geq 1:$%
\begin{eqnarray}
\sum_{j=0}^{\left\lfloor n/2\right\rfloor }\binom{n}{2\left\lfloor
n/2\right\rfloor -2j}2^{2j+n-2n\left\lfloor n/2\right\rfloor
}E_{2j+n-2\left\lfloor n/2\right\rfloor }(x)\allowbreak &=&\allowbreak
(2x-1)^{n},  \label{tozE} \\
\sum_{j=0}^{\left\lfloor n/2\right\rfloor }\binom{n}{2\left\lfloor
n/2\right\rfloor -2j}2^{2j+n-2n\left\lfloor n/2\right\rfloor }\frac{%
B_{2j+n-2\left\lfloor n/2\right\rfloor }(x)}{2j+1}\allowbreak &=&\allowbreak
(2x-1)^{n}  \label{tozB}
\end{eqnarray}
\end{theorem}

\begin{proof}
We start with the following identities: 
\begin{eqnarray*}
\cosh (t/2)\frac{2\exp (xt)}{\exp (t)+1}\allowbreak  &=&\allowbreak \exp
(t(x-1/2)), \\
\frac{t\exp (xt)}{\exp (t)-1}\frac{2\sinh (t/2)}{t}\allowbreak  &=&\exp
(t(x-1/2))\allowbreak .
\end{eqnarray*}
Reacall that we also have: 
\begin{equation*}
\cosh (t/2)\allowbreak =\allowbreak \sum_{j\geq 0}\frac{t^{2j}}{2^{2j}(2j)!},%
\frac{2\sinh (t/2)}{t}\allowbreak =\allowbreak \sum_{j\geq 0}\frac{t^{2j}}{%
2^{2j}(2j)!(2j+1)}.
\end{equation*}%
So applying the standard Cauchy multiplication of two series we get
respectively:%
\begin{eqnarray}
\sum_{n\geq 0}\frac{t^{n}}{n!2^{n}}(2x-1)^{n} &=&\sum_{j\geq 0}\frac{t^{2j}}{%
2^{2j}(2j)!}\sum_{j\geq 0}\frac{t^{j}}{j!}E_{j}(x)\allowbreak   \label{_E} \\
&=&\sum_{n\geq 0}\frac{t^{n}}{n!}\sum_{j=0}^{n}\binom{n}{j}c_{j}E_{n-j}(x),
\\
\sum_{n\geq 0}\frac{t^{n}}{n!2^{n}}(2x-1)^{n}\allowbreak  &=&\allowbreak
\sum_{j\geq 0}\frac{t^{2j}}{2^{2j}(2j)!(2j+1)}\sum_{j\geq 0}\frac{t^{j}}{j!}%
B_{j}(x)\allowbreak   \label{_B} \\
&=&\allowbreak \sum_{n\geq 0}\frac{t^{n}}{n!}\sum_{j=0}^{n}\binom{n}{j}%
c_{j}^{^{\prime }}B_{n-j}(x),
\end{eqnarray}%
where we denoted by $c_{n}$ and $c_{n}^{^{\prime }}$ the following numbers:%
\begin{equation*}
c_{n}=\left\{ 
\begin{array}{ccc}
\frac{1}{2^{n}} & if & n=2\left\lfloor n/2\right\rfloor  \\ 
0 & if & \text{otherwise}%
\end{array}%
\right. ,~c_{n}^{^{\prime }}=\left\{ 
\begin{array}{ccc}
\frac{1}{2^{n}(n+1)} & if & n=2\left\lfloor n/2\right\rfloor  \\ 
0 & if & \text{otherwise}%
\end{array}%
\right. .
\end{equation*}%
Making use of uniqueness of characteristic functions we can equate functions
of $x$ standing by $t^{n}.$ Finally let us multiply both sides so obtained
identities by $2^{n}.$ We have obtained (\ref{tozE}) and (\ref{tozB}).
\end{proof}

We have the following other result:

\begin{theorem}
Let $e(i)\allowbreak =\allowbreak \left\{ 
\begin{array}{ccc}
0 & if & i\text{ is odd} \\ 
1 & if & i\text{ is even}%
\end{array}%
\right. ,$ then%
\begin{eqnarray}
\mathcal{[}e\mathcal{(}i-j)\binom{i}{j}]^{-1} &=&[\binom{i}{j}E_{i-j}],
\label{invE} \\
\mathcal{[}e(i-j)\binom{i}{j}\frac{1}{i-j+1}]^{-1} &=&\mathcal{[}\binom{i}{j}%
\sum_{k=0}^{i-j}\binom{i-j}{k}2^{k}B_{k}].  \label{invB}
\end{eqnarray}
\end{theorem}

\begin{proof}
Let us define by $W_{n}(x)\allowbreak =\allowbreak 2^{n}E_{n}((x+1)/2)$ and $%
V_{n}(x)\allowbreak =\allowbreak 2^{n}B_{n}((x+1)/2).$ Notice that
characteristic function of polynomials $W_{n}$ and $V_{n}$ are given by 
\begin{eqnarray*}
\sum_{j\geq 0}\frac{t^{j}}{j!}W_{j}(x) &=&\sum_{j\geq 0}\frac{(2t)^{j}}{j!}%
E_{j}((x+1)/2) \\
&=&\frac{2\exp (2t(x+1)/2)}{\exp (2t)+1}\allowbreak =\allowbreak \frac{\exp
(tx)}{\cosh (t)}, \\
\sum_{j\geq 0}\frac{t^{j}}{j!}V_{j}(x) &=&\sum_{j\geq 0}\frac{(2t)^{j}}{j!}%
B_{j}((x+1)/2) \\
&=&\frac{\exp (2t(x+1)/2)2t}{\exp (2t)-1}=\frac{t\exp (tx)}{\sinh (t)}
\end{eqnarray*}%
Now recall that $\frac{1}{\cosh (t)}$ is a characteristic function of Euler
numbers while $\frac{t}{\sinh t}$ equal to the characteristic function of
numbers $\left\{ \sum_{j=0}^{n}\binom{n}{j}2^{j}B_{j}\right\} _{n\geq 0}.$
Hence on one hand see that%
\begin{eqnarray*}
W_{n}(x)\allowbreak &=&\allowbreak \sum_{j=0}^{n}\binom{n}{j}x^{n-j}E_{j}, \\
V_{n}(x) &=&\sum_{j=0}^{n}\binom{n}{j}x^{n-j}\sum_{k=0}^{j}\binom{j}{k}%
2^{k}B_{k}
\end{eqnarray*}%
On the other substituting $x$ by $(x+1)/2$ in (\ref{tozE}) and (\ref{tozB})
we see that 
\begin{eqnarray*}
x^{n}\allowbreak &=&\allowbreak \sum_{j=0}^{n}e(n-j)\binom{n}{j}W_{j}(x), \\
x^{n} &=&\sum_{j=0}^{n}e(n-j)\binom{n}{j}\frac{1}{n-j+1}V_{j}
\end{eqnarray*}%
By uniqueness of the polynomial expansion we deduce (\ref{invE}) and (\ref%
{invB}).
\end{proof}

As a corollary we get the following result following also well known
properties of lower triangular matrices (see e.g. : \cite{Hand97}).

\begin{corollary}
\begin{eqnarray*}
\lbrack \binom{2i}{2j}]^{-1} &=&[\binom{2i}{2j}E_{2(i-j)}], \\
\lbrack \binom{2i}{2j}\frac{1}{2(i-j)+1}]^{-1} &=&[\binom{2i}{2j}%
\sum_{k=0}^{2i-2j}\binom{2i-2j}{k}2^{k}B_{k}].
\end{eqnarray*}
\end{corollary}

As in Section \ref{intr} we can multiply both sides of (\ref{tozE}) and (\ref%
{tozB}) by $\lambda ^{n}$ and redefine appropriate vectors and rephrase out
results in terms of modified Pascal matrices.

\begin{corollary}
Forr all $\lambda \in \mathbb{R}$:%
\begin{eqnarray}
\mathcal{[}e\mathcal{(}i-j)\binom{i}{j}\lambda ^{i-j}]^{-1} &=&[\binom{i}{j}%
\lambda ^{i-j}E_{i-j}],  \label{pE} \\
\mathcal{[}e(i-j)\binom{i}{j}\frac{\lambda ^{i-j}}{i-j+1}]^{-1} &=&\mathcal{[%
}\binom{i}{j}\lambda ^{i-j}\sum_{k=0}^{i-j}\binom{i-j}{k}2^{k}B_{k}],
\label{pBB} \\
\lbrack \binom{2i}{2j}\lambda ^{i-j}]^{-1} &=&[\binom{2i}{2j}\lambda
^{i-j}E_{2(i-j)}],  \label{p2E} \\
\lbrack \binom{2i}{2j}\frac{\lambda ^{i-j}}{2(i-j)+1}]^{-1} &=&[\binom{2i}{2j%
}\lambda ^{i-j}\sum_{k=0}^{2i-2j}\binom{2i-2j}{k}2^{k}B_{k}].  \label{p2B}
\end{eqnarray}
\end{corollary}

\end{document}